\documentclass[12pt]{article}
\usepackage{setting0727}
\usepackage{needspace}

\renewcommand{\OpenworkHeaderText}{Zhengxu Jiang, Jiaao Li, B-coloring of grid graphs, }
\title{B-coloring of grid graphs}
\author[1]{Zhengxu Jiang}
\author[1]{Jiaao Li}
\affil[1]{School of Mathematical Sciences and LPMC, Nankai University, Tianjin 300071, China}
\date{}
\hypersetup{
  pdftitle={B-coloring of grid graphs},
  pdfauthor={Zhengxu Jiang and Jiaao Li},
  pdfkeywords={B-coloring, grid graph, cylindrical grid, torus grid, rainbow 4-cycle}
}

\newcommand{\Z}{\mathbb Z}

\definecolor{BcolOne}{HTML}{E6194B}
\definecolor{BcolTwo}{HTML}{0066FF}
\definecolor{BcolThree}{HTML}{00A651}
\definecolor{BcolFour}{HTML}{FF8C00}
\definecolor{BcolFive}{HTML}{B517E8}
\definecolor{BcolSix}{HTML}{00AFC1}
\tikzset{
  bvertex/.style={circle,fill=black,draw=black,line width=.35pt,inner sep=1.55pt},
  edgecode/.style={font=\scriptsize,fill=white,fill opacity=1,text opacity=1,inner sep=.45pt,text=black},
  bedge1/.style={draw=BcolOne,line width=.95pt},
  bedge2/.style={draw=BcolTwo,line width=.95pt},
  bedge3/.style={draw=BcolThree,line width=.95pt},
  bedge4/.style={draw=BcolFour,line width=.95pt},
  bedge5/.style={draw=BcolFive,line width=.95pt},
  bedge6/.style={draw=BcolSix,line width=.95pt}
}

\begin{document}
\maketitle

\begin{abstract}
A B-coloring of a graph $G$ is a proper edge-coloring in which every $4$-cycle is rainbow. Let $q_B(G)$ be the minimum number of colors in such a coloring. Gy\'arf\'as and S\'ark\"ozy (2023) determine $q_B(G)$ when $G=P_m\square P_n$ is a rectangular grid. In this paper, we completely determine $q_B(G)$ for cylindrical and torus grid graphs $G$. For a torus grid $G=C_m\square C_n$, where $m,n\ge3$ are integers, we prove that $q_B(G)=4=\Delta(G)$ if both $m$ and $n$ are even and $G\not\cong C_4\square C_{4k+2}$ for any integer $k\ge1$, and that $q_B(G)=5=\Delta(G)+1$ if at least one of $m,n$ is odd or $G\cong C_4\square C_{4k+2}$ for some integer $k\ge1$. For a cylindrical grid $G=C_s\square P_m$, $q_B(G)$ also depends on the parity of $s$ and the length of $P_m$. For integers $m\ge2$ and $n\ge2$, we have $q_B(C_{2n}\square P_m)=4$. For integers $m\ge2$ and $n\ge1$, we have $q_B(C_{2n+1}\square P_m)=4$ if $2\le m\le n$, whereas $q_B(C_{2n+1}\square P_m)=5$ if $m\ge n+1$. In higher dimensions, we discuss the B-coloring of discrete torus and $\ell$-cylindrical grid, obtaining some exact results and certain bounds.

\end{abstract}

\section{Introduction}
All graphs in this paper are finite and simple. An edge-coloring is \emph{proper} if incident edges receive distinct colors, and a set of edges is \emph{rainbow} if its edges have pairwise distinct colors. A \emph{B-coloring} of a graph $G$ is a proper edge-coloring in which every $4$-cycle is rainbow, and its \emph{B-coloring number} $q_B(G)$ is the minimum number of colors in such a coloring. The properness condition naturally gives $q_B(G)\ge\Delta(G)$, where $\Delta(G)$ is the maximum degree of $G$. Analogously to Vizing's Theorem, we ask when a grid graph attains this bound and when more colors are forced.

Denote $V(G)$ and $E(G)$ to be the vertex and edge sets of a graph $G$. The Cartesian product $G\square H$ has vertex set $V(G)\times V(H)$, with $(u,x)$ adjacent to $(v,y)$ if and only if $uv\in E(G)$ and $x=y$, or $u=v$ and $xy\in E(H)$. Let $P_n$ and $C_n$ denote the path and cycle on $n$ vertices. The products $P_m\square P_n$, $P_m\square C_n$, and $C_m\square C_n$ are the \emph{rectangular}, \emph{cylindrical}, and \emph{torus grids}; passing from paths to cycles closes coordinate directions of the grid.

The requirement that every $4$-cycle be rainbow first appeared in the study of hypercubes~\cite{FaudreeGyarfasLesniakSchelp1993} and was later extended to graph products~\cite{FaudreeGyarfasSchelp1996}. Gy\'arf\'as and S\'ark\"ozy~\cite{GyarfasSarkozyLess2023} subsequently imposed properness, introduced B-coloring, and connected it with the Brown--Erd\H{o}s--S\'os $(7,4)$-conjecture~\cite{BrownErdosSos1973}. The parameter belongs to a broader line of questions about forcing many colors on small subgraphs, including the Ramsey-type problem of Erd\H{o}s and Gy\'arf\'as~\cite{ErdosGyarfas1997}. Star edge coloring is especially close: it is also proper and forbids bichromatic paths and cycles of length four~\cite{DvorakMoharSamal2013}.

Subsequent work has asked how closely $q_B(G)$ follows the degree lower bound. For planar graphs, the current general estimate is $q_B(G)\le2\Delta(G)+6$, and it sharpens to $q_B(G)\le2\Delta(G)$ when $\Delta(G)\ge38$; for outerplanar graphs, $q_B(G)=\Delta(G)$ when $\Delta(G)\ge7$~\cite{GyarfasMartinRuszinkoSarkozy2024,KongWangZheng2026}. Recent results on maximal planar, subcubic, and regular bipartite graphs likewise study the relation between the B-coloring number and the maximum degree $\Delta(G)$~\cite{ChenWangXu2026,XueHuKong2026,GyarfasSarkozyWagner2026}. Cartesian products make this comparison especially concrete because their factor directions determine both the maximum degree and the coordinate $4$-cycles.

Gy\'arf\'as and S\'ark\"ozy~\cite{GyarfasSarkozy2023} began the exact study of Cartesian products of paths and cycles by classifying all path products. For a graph $G$ and a positive integer $d$, let $G^d$ denote the product of $d$ copies of $G$, and let $Q_d=P_2^d$.

\begin{theorem}[Gy\'arf\'as and S\'ark\"ozy~\cite{GyarfasSarkozy2023}]
\label{thm:GS-path-products}
Let $G$ be a Cartesian product of paths. Then $q_B(G)=\Delta(G)$ unless $G$ is isomorphic to $P_2\square P_n$ for some integer $n\ge2$, $Q_3$, or $Q_5$. In the exceptional cases,
\[
q_B(P_2\square P_n)=4,\qquad q_B(Q_3)=4,\qquad q_B(Q_5)=6.
\]
\end{theorem}

Apart from three low-dimensional exceptions, this path-product theorem shows that products of paths attain the degree lower bound $\Delta(G)$. Closing a path direction into a cycle requires the coloring to close consistently as well, introducing the parity and divisibility constraints below.

A product of cycles is a \emph{discrete torus}, and it is \emph{even} when every factor is even. For an integer $d\ge1$, an even discrete torus with factor lengths $2n_1\le\cdots\le2n_d$, where the $n_i\ge2$ are integers, is \emph{exceptional}~\cite{GyarfasSarkozy2023} if $n_1=2$ and $2n_i\equiv2\pmod4$ for $i=2,\ldots,d$. Equivalently, it has the form $C_4\square C_{2n_2}\square\cdots\square C_{2n_d}$ with $n_2,\ldots,n_d$ odd. The known even-cycle result is the following.

\begin{theorem}[Gy\'arf\'as and S\'ark\"ozy~\cite{GyarfasSarkozy2023}]
\label{thm:GS-even-products}
Let $d\ge1$ be an integer, and let $2n_1\le\cdots\le2n_d$ be the factor lengths of an even discrete torus, where the $n_i\ge2$ are integers. If this discrete torus is non-exceptional, then
\[
q_B(C_{2n_1}\square\cdots\square C_{2n_d})=2d.
\]
If it is exceptional and $d\ge2$, then
\[
2d\le q_B(C_{2n_1}\square\cdots\square C_{2n_d})\le2d+1.
\]
Moreover, $q_B(C_4)=4$ and $q_B(C_4\square C_{4k+2})=5$ for every integer $k\ge1$.
\end{theorem}

Since a $d$-dimensional discrete torus is $2d$-regular, Theorem~\ref{thm:GS-even-products} shows that the non-exceptional case attains the degree lower bound. The exceptional values were known for $d=1$ and $d=2$, but not in higher dimensions. Products of odd cycles require at least $2d+1$ colors even for proper edge-coloring; the known bounds are as follows.

\begin{theorem}[Gy\'arf\'as and S\'ark\"ozy~\cite{GyarfasSarkozy2023}]
\label{thm:GS-odd-products}
Let $d\ge1$ and $n_1,\ldots,n_d\ge1$ be integers, and set $G=C_{2n_1+1}\square\cdots\square C_{2n_d+1}$. Then $2d+1\le q_B(G)\le3d$. If every factor length is divisible by $2d+1$, then $q_B(G)=2d+1$. Furthermore,
\[
2d+1\le q_B(C_3^d)\le
\begin{cases}
5d/2,&d\text{ is even},\\
(5d+1)/2,&d\text{ is odd}.
\end{cases}
\]
\end{theorem}

Theorems~\ref{thm:GS-even-products} and~\ref{thm:GS-odd-products} therefore leave exceptional even discrete tori open in higher dimensions and odd products open beyond the divisibility case. We first settle the two-dimensional odd case, thereby completing the classification of torus grids.

\begin{theorem}[Torus grids]
\label{thm:torus-grids}
For all integers $m,n\ge3$,
\[
q_B(C_m\square C_n)=
\begin{cases}
5,&\text{if at least one of }m,n\text{ is odd};\\
5,&\text{if }\{m,n\}=\{4,4k+2\}\text{ for some integer }k\ge1;\\
4,&\text{otherwise}.
\end{cases}
\]
\end{theorem}

Since torus grids are $4$-regular, Theorem~\ref{thm:torus-grids} gives $q_B(G)=\Delta(G)$ precisely in the non-exceptional even case and $q_B(G)=\Delta(G)+1$ otherwise. Closing only one direction leaves path boundaries: they permit four colors for every even circumference and, for circumference $2n+1$, until the path has more than $n$ vertices.

\begin{theorem}[Cylindrical grids]
\label{thm:cylindrical-grids}
Let $m\ge2$ be an integer. For every integer $n\ge2$,
\[
q_B(C_{2n}\square P_m)=4.
\]
For every integer $n\ge1$,
\[
q_B(C_{2n+1}\square P_m)=
\begin{cases}
4,&2\le m\le n,\\
5,&m\ge n+1.
\end{cases}
\]
\end{theorem}

We next return to the exceptional even family as a continuation of Theorem~\ref{thm:GS-even-products}. Although $C_4\square C_{4k+2}$ requires an extra color for every integer $k\ge1$, the obstruction disappears in higher dimensions.

\begin{theorem}[Exceptional even discrete tori]
\label{thm:exceptional-even-classification}
Let $d\ge1$ be an integer, and let $2n_1\le\cdots\le2n_d$ be the factor lengths of an exceptional even discrete torus, where the $n_i\ge2$ are integers. Then
\[
q_B(C_{2n_1}\square\cdots\square C_{2n_d})=
\begin{cases}
4,&d=1,\\
5,&d=2,\\
2d,&d\ge3.
\end{cases}
\]
\end{theorem}

Finally, for a positive integer $\ell$, an \emph{$\ell$-cylindrical grid} is a Cartesian product of paths and cycles with at least one path factor and exactly $\ell$ cycle factors~\cite{Arora2025}; it is \emph{even} when all cycle factors are even. A cylindrical grid is the two-dimensional $1$-cylindrical case.

To state the classification without duplicate representations, we absorb every pair $P_2\square P_2$ into a $C_4$ factor and denote the path and cycle parts as
\[
P=P_2^r\square P_{m_1}\square\cdots\square P_{m_s},\qquad
E=C_{2n_1}\square\cdots\square C_{2n_t},
\]
where $r,s,t$ and all $m_i,n_i$ are integers, $r\in\{0,1\}$, $s\ge0$, $t\ge1$, $m_i\ge3$ for $i=1,2,\ldots,s$, and $n_i\ge2$ for $i=1,2,\ldots,t$, with $2n_1\le\cdots\le2n_t$. Then $G=P\square E$ is $t$-cylindrical, and $r+s\ge1$ ensures a path factor. Let $c_4(E)$ count the $C_4$ factors of $E$. The next theorem agrees with Theorem~\ref{thm:cylindrical-grids} in two dimensions; in higher dimensions it leaves only three classes with a one-color gap.

\Needspace{18\baselineskip}
\begin{theorem}[Even $\ell$-cylindrical grids]
\label{thm:path-even-cycle-products}
Let $G=P\square E$ be as above, and assume $r+s\ge1$.
\begin{enumerate}
\item If $P\cong P_2$ and $t=1$, then $q_B(G)=4$.
\item If $P\cong P_2$ and $E\cong C_4\square C_4$, then $q_B(G)=6$.
\item In each of the following cases, $\Delta(G)\le q_B(G)\le\Delta(G)+1$:
\begin{enumerate}
\item $P\cong P_2$, $c_4(E)=0$, and $t\ge2$;
\item $P\cong P_2$ and $E\cong C_4\square C_4\square C_{4n+2}$ for some integer $n\ge1$;
\item $P\cong P_2\square P_m$ for some integer $m\ge3$ and $c_4(E)=0$.
\end{enumerate}
\item In all remaining cases, $q_B(G)=\Delta(G)$.
\end{enumerate}
\end{theorem}

The rest of this paper is organized as follows. Section~\ref{sec:preliminaries} records the product tools. Sections~\ref{sec:two-cycles} and~\ref{sec:cylindrical-grids} prove the two-dimensional results; Sections~\ref{sec:exceptional-even} and~\ref{sec:path-products} treat the higher-dimensional classes. The final section states two open problems.

\section{Preliminaries}
\label{sec:preliminaries}
For integers $d\ge1$ and $n_1,\ldots,n_d\ge2$, the product $P_{n_1}\square\cdots\square P_{n_d}$ is a $d$-dimensional grid. We use product coordinates throughout the constructions. Given graphs $G$ and $H$ and a vertex $x\in V(H)$, let $G(x)$ denote the copy of $G$ in $G\square H$ indexed by $x$, and define copies of $H$ analogously. Coordinates belonging to a cycle factor are read modulo the length of that cycle.

The cross inequality below combines B-colorings of two factors, but its hypotheses also involve their chromatic numbers. Sabidussi's theorem supplies these chromatic numbers for all products considered here; as usual, $\chi(G)$ denotes the chromatic number of $G$.

\begin{theorem}[Sabidussi~\cite{Sabidussi1957}]
\label{thm:sabidussi}
For all graphs $G$ and $H$,
\[
\chi(G\square H)=\max\{\chi(G),\chi(H)\}.
\]
\end{theorem}

We use the following proper edge-coloring form of the product construction in~\cite{FaudreeGyarfasLesniakSchelp1993}; see also~\cite{GyarfasSarkozy2023}.

\begin{lemma}[Cross inequality]
\label{lem:cross-inequality}
Suppose that $G$ has a $p$-color B-coloring and that $H$ has an $r$-color B-coloring. If $\chi(G)\le r$ and $\chi(H)\le p$, then
\[
q_B(G\square H)\le p+r.
\]
\end{lemma}

\begin{proof}
Use disjoint color sets $A$ and $B$, with $|A|=p$ and $|B|=r$, and fix B-colorings $\alpha$ of $G$ and $\beta$ of $H$. Identify each palette with a cyclic group, and fix proper vertex colorings of $G$ and $H$ using $\chi(G)$ and $\chi(H)$ colors, respectively. For the vertex colors of $H$, choose distinct cyclic shifts of $A$; this is possible because $\chi(H)\le p$. For the vertex colors of $G$, choose distinct cyclic shifts of $B$; this is possible because $\chi(G)\le r$.

Color the copy $G(x)$ by the shift indexed by the vertex color of $x\in V(H)$, and color copies of $H$ analogously. The coloring is proper because the two factor directions use disjoint palettes and each copy receives a permutation of a proper edge-coloring. A $4$-cycle contained in a single factor copy is rainbow. Every other $4$-cycle is a product square. Its two $G$-edges are copies of the same edge of $G$ indexed by adjacent vertices of $H$, and hence receive distinct colors; the same holds for its two $H$-edges. Since $A$ and $B$ are disjoint, the square is rainbow.
\end{proof}

For products containing a $Q_3$ factor, the following refinement saves one color over the general cross inequality.

\begin{lemma}[The $Q_3$ product lemma~\cite{FaudreeGyarfasLesniakSchelp1993,GyarfasSarkozy2023}]
\label{lem:Q3-product}
If $H$ is a connected bipartite graph with at least two edges, then
\[
q_B(Q_3\square H)\le q_B(H)+3.
\]
Consequently, if $q_B(H)=\Delta(H)$, then $q_B(Q_3\square H)=\Delta(Q_3\square H)$.
\end{lemma}

\begin{proof}
The inequality is the special recoloring construction for $Q_3$ from~\cite{FaudreeGyarfasLesniakSchelp1993}; its proper version is stated in~\cite{GyarfasSarkozy2023}. If $q_B(H)=\Delta(H)$, the upper bound is $\Delta(H)+3$, which equals $\Delta(Q_3\square H)$ and hence is also a lower bound.
\end{proof}

We will repeatedly enlarge cycle factors by repeating an existing B-coloring. The following lemma shows that this can be done in several cycle factors at once.
\begin{lemma}
\label{lem:cover-pullback}
Let $H$ be a graph, let $d$ and $k$ be positive integers, and let $n_i\ge3$ and $t_i\ge1$ be integers for $i=1,2,\ldots,d$. Set
\[
\begin{aligned}
G&=H\square C_{n_1}\square\cdots\square C_{n_d},\\
\widetilde G&=H\square C_{t_1n_1}\square\cdots\square C_{t_dn_d}.
\end{aligned}
\]
If $c$ is a B-coloring of $G$ using $k$ colors, then repeating the color pattern of $c$ $t_i$ times in the $i$th cycle coordinate, for $i=1,2,\ldots,d$, gives a B-coloring of $\widetilde G$ using the same $k$ colors.
\end{lemma}

\begin{proof}
Write a vertex of $\widetilde G$ as $(x,i_1,\ldots,i_d)$, where $x\in V(H)$ and $i_j\in\Z_{t_jn_j}$ for $j=1,2,\ldots,d$. Let $\bar i_j$ be the reduction of $i_j$ modulo $n_j$, and define
\[
\rho(x,i_1,\ldots,i_d)=(x,\bar i_1,\ldots,\bar i_d)\in V(G).
\]
For every edge $uv\in E(\widetilde G)$, the vertices $\rho(u)$ and $\rho(v)$ are adjacent in the same factor direction. Define
\[
\widetilde c(uv)=c\bigl(\rho(u)\rho(v)\bigr).
\]
Thus, in the $j$th cyclic coordinate, the color pattern of period $n_j$ is repeated $t_j$ times.

At each vertex of $\widetilde G$, the map $\rho$ sends the incident edges bijectively to the incident edges at its image in $G$. Their colors under $\widetilde c$ are therefore pairwise distinct, so $\widetilde c$ is proper.

It remains to check the $4$-cycles. A $4$-cycle in a Cartesian product either lies in a single factor or is a product square determined by two distinct factors. If all four edges belong to a copy of $H$, then $\rho$ maps the cycle bijectively onto the corresponding $4$-cycle of $G$. A cycle factor contains a $4$-cycle only when that factor is $C_4$. Since $n_j\ge3$ and $t_jn_j=4$ imply $n_j=4$ and $t_j=1$, such a factor is unchanged. Finally, $\rho$ maps every product square, together with its four boundary edges, bijectively onto a product square of $G$. In every case, the four colors are those of a $4$-cycle of $G$ and hence are pairwise distinct.

Thus $\widetilde c$ is a B-coloring. Every edge of $G$ has a preimage under $\rho$, so every color used by $c$ also occurs under $\widetilde c$. Hence the two colorings use the same $k$ colors.
\end{proof}
Figure~\ref{fig:cover-pullback} shows a five-color B-coloring of $C_3\square C_5$ and the coloring of $C_6\square C_5$ obtained by repeating the $C_3$-direction twice.

\begin{figure}[htbp]
\centering
\begin{tikzpicture}[x=1.0cm,y=0.92cm,scale=1.25,transform shape]

\foreach \x in {0,1,2}{
  \foreach \y in {0,...,4}{
    \node[bvertex] (A\x\y) at (\x,\y) {};
  }
}

\foreach \y/\ca/\cb/\cc in {0/4/5/1,1/5/3/2,2/3/4/1,3/4/1/2,4/5/2/4}{
  \draw[bedge\ca] (A0\y)--(A1\y);
  \draw[bedge\cb] (A1\y)--(A2\y);
  \draw[bedge\cc,bend left=28] (A2\y) to (A0\y);
}

\foreach \x/\cA/\cB/\cC/\cD/\cE in {0/3/4/5/1/2,1/2/1/2/3/1,2/4/5/3/5/3}{
  \draw[bedge\cA] (A\x0)--(A\x1);
  \draw[bedge\cB] (A\x1)--(A\x2);
  \draw[bedge\cC] (A\x2)--(A\x3);
  \draw[bedge\cD] (A\x3)--(A\x4);
  \draw[bedge\cE,bend right=28] (A\x4) to (A\x0);
}

\node[font=\small] at (1,-1.25) {$C_3\square C_5$};

\begin{scope}[xshift=6.2cm]
\foreach \x in {0,...,5}{
  \foreach \y in {0,...,4}{
    \node[bvertex] (B\x\y) at (\x,\y) {};
  }
}

\foreach \y/\ca/\cb/\cc in {0/4/5/1,1/5/3/2,2/3/4/1,3/4/1/2,4/5/2/4}{
  \draw[bedge\ca] (B0\y)--(B1\y);
  \draw[bedge\cb] (B1\y)--(B2\y);
  \draw[bedge\cc] (B2\y)--(B3\y);
  \draw[bedge\ca] (B3\y)--(B4\y);
  \draw[bedge\cb] (B4\y)--(B5\y);
  \draw[bedge\cc,bend left=28] (B5\y) to (B0\y);
}

\foreach \x/\cA/\cB/\cC/\cD/\cE in {0/3/4/5/1/2,1/2/1/2/3/1,2/4/5/3/5/3,3/3/4/5/1/2,4/2/1/2/3/1,5/4/5/3/5/3}{
  \draw[bedge\cA] (B\x0)--(B\x1);
  \draw[bedge\cB] (B\x1)--(B\x2);
  \draw[bedge\cC] (B\x2)--(B\x3);
  \draw[bedge\cD] (B\x3)--(B\x4);
  \draw[bedge\cE,bend right=28] (B\x4) to (B\x0);
}

\node[font=\small] at (2.5,-1.25) {$C_6\square C_5$};
\end{scope}

\draw[-{Latex[length=3mm]},thick]
  (3.15,2.05)--(5.05,2.05)
  node[midway,above,font=\small] {repeat twice};

\node[font=\scriptsize,anchor=west] at (11.4,4.45) {colors};
\foreach \c/\legendy in {1/4.05,2/3.60,3/3.15,4/2.70,5/2.25}{
  \draw[bedge\c] (11.4,\legendy)--++(0.42,0)
    node[right,font=\scriptsize,text=black] {\c};
}

\end{tikzpicture}
\par\smallskip
\caption{A five-color B-coloring of $C_3\square C_5$ and the coloring of $C_6\square C_5$ obtained by repeating the $C_3$-direction twice. The legend at upper right identifies the five edge colors.}
\label{fig:cover-pullback}
\end{figure}
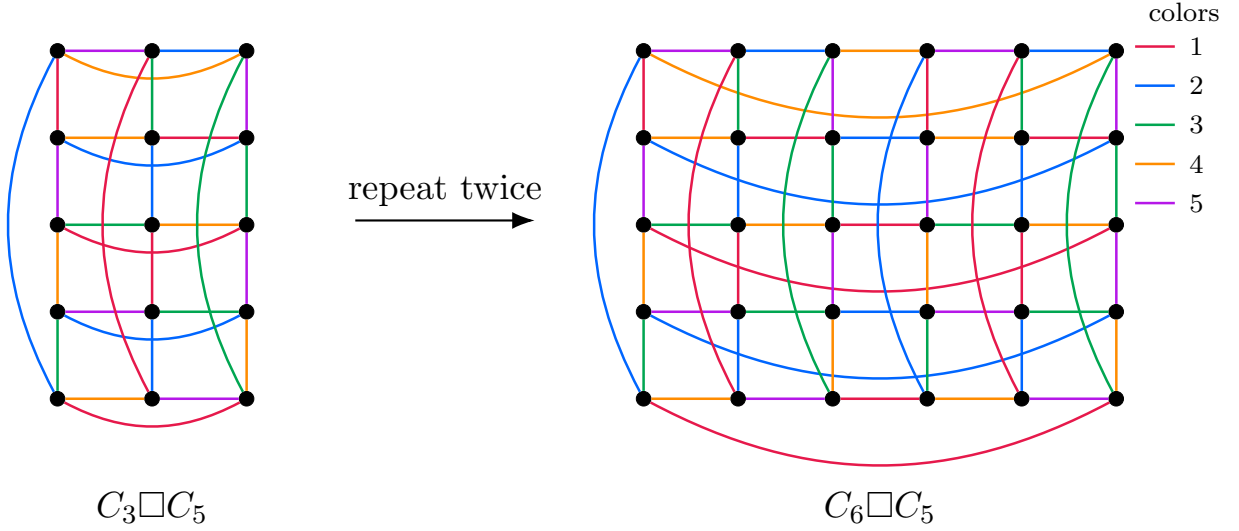

\section{Torus grids with an odd factor}
\label{sec:two-cycles}

\begin{theorem}[Torus grids with an odd side]
\label{thm:cycle-products}
For all integers $m\ge3$ and $n\ge1$,
\[
q_B(C_m\square C_{2n+1})=5.
\]
\end{theorem}

The proof of Theorem~\ref{thm:cycle-products} begins with a parity count that rules out a four-color B-coloring.

\begin{lemma}
\label{lem:cycle-lower-bound}
For all integers $m\ge3$ and $n\ge1$, the torus grid $C_m\square C_{2n+1}$ has no four-color B-coloring.
\end{lemma}

\begin{proof}
Suppose that such a coloring exists. Use coordinates $\Z_m\times\Z_{2n+1}$, and denote by $h(i,j)$ and $v(i,j)$ the colors of $(i,j)(i+1,j)$ and $(i,j)(i,j+1)$, respectively. Fix a color $c$, and for $i\in\Z_m$ let
\[
A_i=|\{j:h(i,j)=c\}|,\qquad B_i=|\{j:v(i,j)=c\}|.
\]
At every vertex, the four incident edges have the four colors. Summing over column $i$ gives
\[
A_{i-1}+A_i+2B_i=2n+1. \tag{3.1}
\]
Similarly, every square contains color $c$ exactly once, and summing over the squares between columns $i$ and $i+1$ gives
\[
2A_i+B_i+B_{i+1}=2n+1. \tag{3.2}
\]
Comparing (3.1) and (3.2), with indices shifted when necessary, yields $A_i+B_i=A_{i-1}+B_{i-1}$. Thus $A_i+B_i=T$ is independent of $i$. Substituting $B_i=T-A_i$ and $B_{i+1}=T-A_{i+1}$ into (3.2) gives $A_{i+1}-A_i=2T-(2n+1)$. Since the indices are cyclic, this constant difference is zero. Hence $2T=2n+1$, a contradiction.
\end{proof}

It remains to construct five-colorings. We begin with $C_4\square C_{2n+1}$.

\begin{table}[htbp]
\centering
\caption{Blocks for the five-coloring of $C_4\square C_{2n+1}$. In each column, $Y_i$ lists the colors on the $C_4$ copy and $X_i$ lists the four edges to the next column.}
\label{tab:C4-odd-blocks}
\begin{tabular}{c|cc|cc|cc}
\toprule
& \multicolumn{2}{c|}{$B_3$} & \multicolumn{2}{c|}{$B_4$} & \multicolumn{2}{c}{$B_5$}\\
column & $Y_i$ & $X_i$ & $Y_i$ & $X_i$ & $Y_i$ & $X_i$\\
\midrule
0 & 1234 & 3542 & 1234 & 5451 & 1234 & 5451\\
1 & 2315 & 4153 & 2143 & 1532 & 2143 & 4532\\
2 & 5421 & 2315 & 3214 & 5143 & 3251 & 5413\\
3 &      &      & 4521 & 2315 & 1342 & 4251\\
4 &      &      &      &      & 5423 & 2315\\
\bottomrule
\end{tabular}
\end{table}

\begin{lemma}
\label{lem:C4-odd-cycle}
For every integer $n\ge1$, $q_B(C_4\square C_{2n+1})\le5$.
\end{lemma}

\begin{proof}
Set $q=2n+1$ and use coordinates $(i,j)\in\Z_q\times\Z_4$, where $i$ indexes the $C_q$ factor.
If $q\equiv3\pmod4$, concatenate one copy of $B_3$ and $(q-3)/4$ copies of $B_4$. If $q\equiv1\pmod4$, concatenate one copy of $B_5$ and $(q-5)/4$ copies of $B_4$. The final horizontal list of each block is $2315$, and the initial vertical list of each block is $1234$, so the same local transition occurs at every concatenation point.

Let $y_i(j)$ and $x_i(j)$ be the $j$th entries of $Y_i$ and $X_i$, respectively. At vertex $(i,j)$ the four incident colors are
\[
\{y_i(j-1),y_i(j),x_{i-1}(j),x_i(j)\},
\]
and the square with lower-left corner $(i,j)$ has colors
\[
\{x_i(j),x_i(j+1),y_i(j),y_{i+1}(j)\}.
\]
Substitution in Table~\ref{tab:C4-odd-blocks} shows that both sets have size four, including at every block boundary. Hence the concatenated coloring is a B-coloring.
\end{proof}

\begin{corollary}
\label{cor:C4a-odd-cycle}
For all integers $n_1,n_2\ge1$, $q_B(C_{4n_1}\square C_{2n_2+1})\le5$.
\end{corollary}

\begin{proof}
By Lemma~\ref{lem:cover-pullback}, repeating the $C_4$-direction of the coloring in Lemma~\ref{lem:C4-odd-cycle} $n_1$ times gives a five-color B-coloring of $C_{4n_1}\square C_{2n_2+1}$.
\end{proof}

For two odd factors, we use an explicit seam construction. The full seam data are given in Appendix~\ref{app:odd-seams}; only the defining formulas are needed in the proof below.

\begin{proposition}
\label{prop:odd-odd-upper}
For all integers $n_1,n_2\ge1$,
\[
q_B(C_{2n_1+1}\square C_{2n_2+1})\le5.
\]
\end{proposition}

\begin{proof}
Set $p=2n_1+1$ and $q=2n_2+1$, and use coordinates $(i,j)\in\Z_p\times\Z_q$. Let $h(i,j)$ and $v(i,j)$ denote the colors of $(i,j)(i+1,j)$ and $(i,j)(i,j+1)$, respectively.

Suppose first that $q=3$. For $0\le i\le p-3$, define
\[
\begin{aligned}
h(i,0)&=(i\bmod3)+3, & v(i,0)&=((i+1)\bmod3)+3,\\
h(i,1)&=2-(i\bmod2), & v(i,1)&=((i+2)\bmod3)+3,\\
h(i,2)&=((i+1)\bmod3)+3, & v(i,2)&=1+(i\bmod2).
\end{aligned}
\]
Let $H_i=(h(i,0),h(i,1),h(i,2))$ and $V_i=(v(i,0),v(i,1),v(i,2))$. The last two columns are specified in Table~\ref{tab:q3-seam}.

\begin{table}[htbp]
\centering
\caption{The final two columns in the construction of $C_{2n_1+1}\square C_3$.}
\label{tab:q3-seam}
\begin{tabular}{c|cccc}
\toprule
$p\bmod6$ & $H_{p-2}$ & $V_{p-2}$ & $H_{p-1}$ & $V_{p-1}$\\
\midrule
1 & 153 & 342 & 532 & 214\\
3 & 135 & 512 & 512 & 243\\
5 & 135 & 412 & 213 & 524\\
\bottomrule
\end{tabular}
\end{table}
Checking the incident colors and the adjacent squares at the two seam columns shows that the coloring is proper and that every square is rainbow. Interchanging the two factors also covers $p=3$.

For $p=q=5$, identify the colors with $\Z_5$ and define $h(i,j)=i+j$ and $v(i,j)=i+j+2$. At each vertex the four incident colors are $i+j-1,i+j,i+j+1,i+j+2$, and on each square they are $i+j,i+j+1,i+j+2,i+j+3$.

After interchanging the factors if necessary, the only remaining range is $p\ge5$ and $q\ge7$. Begin with
\[
h_0(i,j)=4+((i+j+1)\bmod2),\qquad
v_0(i,j)=((j+2-(i\bmod2))\bmod3)+1.
\]
Modify the terminal rows as listed in Appendix~\ref{app:odd-seams}. Then modify only columns $0,p-2,p-1$. For $\xi\in\{0,p-2,p-1\}$, denote by $H_j^\xi$ and $V_j^\xi$ the replacement colors assigned to $h(\xi,j)$ and $v(\xi,j)$, respectively, and collect them as
\[
T_j=(H_j^0,H_j^{p-2},H_j^{p-1};V_j^0,V_j^{p-2},V_j^{p-1}),
\]
the six entries prescribe, respectively, $h(0,j),h(p-2,j),h(p-1,j)$ and $v(0,j),v(p-2,j),v(p-1,j)$. The lists $T_j$ are also given in Appendix~\ref{app:odd-seams}.

\Needspace{10\baselineskip}
All remaining entries retain the values $h_0,v_0$. To verify the construction, it remains to check
\[
|\{h(i-1,j),h(i,j),v(i,j-1),v(i,j)\}|=4, \tag{3.3}
\]
at every vertex and
\[
|\{h(i,j),h(i,j+1),v(i,j),v(i+1,j)\}|=4. \tag{3.4}
\]
for every square.
Away from the terminal rows and columns, the periodic formulas give (3.3)--(3.4). At a seam, each check involves at most two consecutive rows and two consecutive columns. Substituting the finite lists in Appendix~\ref{app:odd-seams} verifies (3.3)--(3.4), including the two cyclic closures. Thus the displayed rules give a five-color B-coloring of $C_{2n_1+1}\square C_{2n_2+1}$.
\end{proof}

\begin{proof}[Proof of Theorem~\ref{thm:cycle-products}]
Lemma~\ref{lem:cycle-lower-bound} gives the lower bound five. For the matching upper bound, Proposition~\ref{prop:odd-odd-upper} applies when $m$ is odd, and Corollary~\ref{cor:C4a-odd-cycle} applies when $m\equiv0\pmod4$. In the remaining case, $m=4a+2$ for some integer $a\ge1$. Proposition~\ref{prop:odd-odd-upper} gives a five-color B-coloring of $C_{2a+1}\square C_{2n+1}$, and Lemma~\ref{lem:cover-pullback} repeats its $C_{2a+1}$-direction twice to give the required coloring of $C_{4a+2}\square C_{2n+1}$.
\end{proof}

\begin{proof}[Proof of Theorem~\ref{thm:torus-grids}]
If at least one side length is odd, the result is Theorem~\ref{thm:cycle-products}. Suppose that $m$ and $n$ are even, with $m\le n$. When $m=4$ and $n=4k+2$ for some integer $k\ge1$, Theorem~\ref{thm:GS-even-products} gives $q_B(C_m\square C_n)=5$. Every other even--even torus grid is non-exceptional, so the same theorem gives $q_B(C_m\square C_n)=4$.
\end{proof}

Theorem~\ref{thm:cycle-products} implies the following corollary.

\begin{corollary}
\label{cor:many-odd-upper}
Let $d\ge1$ and $n_1,\ldots,n_d\ge1$ be integers. Then
\[
q_B(C_{2n_1+1}\square\cdots\square C_{2n_d+1})
\le\left\lceil\frac{5d}{2}\right\rceil.
\]
\end{corollary}

\begin{proof}
Pair the first $2\lfloor d/2\rfloor$ factors. By Theorem~\ref{thm:cycle-products}, each pair has a five-color B-coloring; any unpaired odd cycle has a three-color B-coloring. Theorem~\ref{thm:sabidussi} gives chromatic number three for every nonempty product of these factors, so the hypotheses of Lemma~\ref{lem:cross-inequality} hold at each step. Repeated application of the lemma uses five colors per pair and three for the unpaired cycle, when present, for a total of $\lceil5d/2\rceil$ colors.
\end{proof}

\section{Cylindrical grids}
\label{sec:cylindrical-grids}

\subsection{Odd circumference}

\begin{theorem}[Odd cylindrical grids]
\label{thm:odd-cycle-path}
Let $n\ge1$ and $m\ge2$ be integers. Then
\[
q_B(C_{2n+1}\square P_m)=
\begin{cases}
4,&2\le m\le n,\\
5,&m\ge n+1.
\end{cases}
\]
\end{theorem}

Identify $V(C_{2n+1}\square P_m)$ with $\Z_{2n+1}\times\{0,1,\ldots,m-1\}$. In an edge-coloring, let $h_j(i)$ denote the color of $(i,j)(i+1,j)$ for $i\in\Z_{2n+1}$ and $0\le j\le m-1$, and let $v_j(i)$ denote the color of $(i,j)(i,j+1)$ for $i\in\Z_{2n+1}$ and $0\le j\le m-2$.

\begin{lemma}
\label{lem:odd-path-short}
If $m$ and $n$ are integers satisfying $n\ge2$ and $2\le m\le n$, then $C_{2n+1}\square P_m$ has a four-color B-coloring.
\end{lemma}

\begin{proof}
Identify the four colors with $\Z_2^2=\{0,1,2,3\}$, with addition interpreted bitwise. Define
\[
v_0(i)=
\begin{cases}
0,&0\le i\le2n-2\text{ and }i\text{ is even},\\
1,&0\le i\le2n-2\text{ and }i\text{ is odd},\\
2,&i=2n-1,\\
3,&i=2n,
\end{cases}
\]
and
\[
h_0(i)=
\begin{cases}
3,&0\le i\le2n-2\text{ and }i\text{ is even},\\
2,&0\le i\le2n-2\text{ and }i\text{ is odd},\\
1,&i=2n-1,\\
2,&i=2n.
\end{cases}
\]
Recursively set
\[
h_{j+1}(i)=h_j(i)+v_j(i)+v_j(i+1), \tag{4.1}
\]
and, whenever another vertical layer is needed,
\[
v_{j+1}(i)=v_j(i)+h_{j+1}(i-1)+h_{j+1}(i). \tag{4.2}
\]
Since the sum of three distinct elements of $\Z_2^2$ is the fourth element, (4.1) makes every new square rainbow and (4.2) makes every new interior vertex incident with four distinct colors.

The recurrence also gives the following explicit formulas. For $0\le j\le n-1$,
\[
v_j(i)=
\begin{cases}
2+((i+j)\bmod2),&0\le i<j,\\
(i-j)\bmod2,&j\le i\le2n-j-2,\\
2+((i-2n+j+1)\bmod2),&2n-j-1\le i\le2n,
\end{cases} \tag{4.3}
\]
and, for $1\le j\le n-1$,
\[
h_j(i)=
\begin{cases}
(i+j+1)\bmod2,&0\le i\le j-2,\\
2,&i=j-1,\\
3-((i-j)\bmod2),&j\le i\le2n-j-2,\\
1-((i-2n+j+1)\bmod2),&2n-j-1\le i\le2n.
\end{cases} \tag{4.4}
\]
On each alternating interval, (4.1)--(4.2) preserve the displayed pattern. Substitution at the four breakpoints verifies the exceptional entries and completes the induction. Equations (4.3)--(4.4) also show that both boundary rows are properly colored. Restricting the construction to the first $m$ rows proves the lemma.
\end{proof}

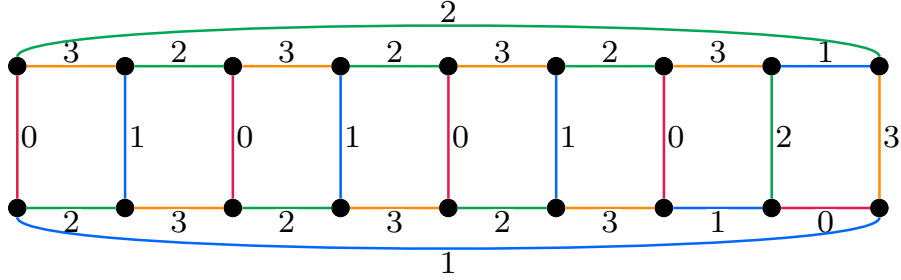
\begin{figure}[htbp]
\centering
\begin{tikzpicture}[x=0.95cm,y=1.25cm,scale=1.5,transform shape]
\foreach \i in {0,...,8}{\node[bvertex] (u\i) at (\i,1) {}; \node[bvertex] (v\i) at (\i,0) {};}
\foreach \i/\c in {0/4,1/3,2/4,3/3,4/4,5/3,6/4,7/2}{\draw[bedge\c] (u\i)--node[edgecode,above]{\pgfmathparse{int(\c-1)}\pgfmathresult} (u\the\numexpr\i+1\relax);}
\draw[bedge3] (u8)
  .. controls (8,1.36) and (0,1.36)
  .. node[edgecode,above]{2} (u0);
\foreach \i/\c in {0/3,1/4,2/3,3/4,4/3,5/4,6/2,7/1}{\draw[bedge\c] (v\i)--node[edgecode,below]{\pgfmathparse{int(\c-1)}\pgfmathresult} (v\the\numexpr\i+1\relax);}
\draw[bedge2] (v8)
  .. controls (8,-0.36) and (0,-0.36)
  .. node[edgecode,below]{1} (v0);
\foreach \i/\c in {0/1,1/2,2/1,3/2,4/1,5/2,6/1,7/3,8/4}{\draw[bedge\c] (u\i)--node[edgecode,right]{\pgfmathparse{int(\c-1)}\pgfmathresult} (v\i);}
\end{tikzpicture}
\caption{The initial two rows of the four-color construction for $C_9\square P_m$, where $m$ is an integer satisfying $2\le m\le4$. Edge colors are shown both by line color and by the labels $0,1,2,3$.}
\label{fig:odd-path-strip}
\end{figure}

\Needspace{8\baselineskip}
\begin{lemma}
\label{lem:odd-path-long-lower}
If $m$ and $n$ are integers satisfying $n\ge2$ and $m\ge n+1$, then $C_{2n+1}\square P_m$ has no four-color B-coloring.
\end{lemma}

\begin{proof}
Suppose that a four-color B-coloring exists. For a color $c$, let $A_j(c)$ be the number of horizontal edges of color $c$ in row $j$, and let $B_j(c)$ be the number of vertical edges of color $c$ between rows $j$ and $j+1$. Fix $c$ and abbreviate $A_j=A_j(c)$ and $B_j=B_j(c)$.

Counting color $c$ on the squares between two consecutive rows gives, for $0\le j\le m-2$,
\[
A_j+A_{j+1}+2B_j=2n+1, \tag{4.5}
\]
while counting it on the edges incident with the vertices of an interior row gives, for $1\le j\le m-2$,
\[
2A_j+B_{j-1}+B_j=2n+1. \tag{4.6}
\]

Comparing (4.5) for $j-1$ with (4.6) for $j$ shows that $A_j+B_j$ is independent of $j$. Comparing (4.5) with (4.6) in the other direction shows that $A_{j+1}+B_j$ is also independent of $j$. The difference of these two constants is $A_{j+1}-A_j$, and hence $A_0,A_1,\ldots,A_{m-1}$ is an arithmetic progression. Its common difference is odd, since the sum of the constants is $2n+1$.

The edges of color $c$ in a cycle row form a matching, so $0\le A_j\le n$. Since the progression has at least $n+1$ terms, its odd common difference must be $1$ or $-1$, and the same bounds force $A_0(c)\in\{0,n\}$. Summing over the four colors in the first row gives a multiple of $n$, but that row has $2n+1$ edges, a contradiction.
\end{proof}

\begin{lemma}
\label{lem:C3-path}
For every integer $m\ge2$, $C_3\square P_m$ has no four-color B-coloring.
\end{lemma}

\begin{proof}
It suffices to consider $C_3\square P_2$. Relabel the colors so that one triangular layer has edge colors $1,2,3$ in cyclic order. Properness restricts the three joining edges to the sets $\{2,4\}$, $\{3,4\}$, and $\{1,4\}$, respectively. The square conditions leave only the joining triples
\[
(2,3,1),\ (2,3,4),\ (2,4,1),\ (4,3,1).
\]
The corresponding forced color triples on the second triangle are, respectively,
\[
(4,4,4),\ (4,1,1),\ (3,3,4),\ (2,4,2),
\]
none of which is proper.
\end{proof}

\begin{proof}[Proof of Theorem~\ref{thm:odd-cycle-path}]
When $2\le m\le n$, Lemma~\ref{lem:odd-path-short} gives the upper bound four, and the graph contains a $4$-cycle. If $m\ge n+1$ and $n\ge2$, Lemma~\ref{lem:odd-path-long-lower} gives the lower bound five; the case $n=1$ follows from Lemma~\ref{lem:C3-path}. For the upper bound, embed $P_m$ in any odd cycle of length at least $m+1$ and restrict the five-coloring supplied by Theorem~\ref{thm:cycle-products}.
\end{proof}

\subsection{Even circumference}

\begin{lemma}
\label{lem:P2-even-cycle}
For every integer $n\ge2$, $q_B(P_2\square C_{2n})=4$.
\end{lemma}

\begin{proof}
Let $N=2n$ and use coordinates $\{0,1\}\times\Z_N$. For $i\in\Z_N$, let $d_i$ be the color of the $P_2$-edge in column $i$, and let $x_i,y_i$ be the colors of the two cycle edges from column $i$ to column $i+1$. Define
\[
(d_i,x_i,y_i)=
\begin{cases}
(1,2,3),&i=0,\\
(4,1,2),&i\text{ is odd and }1\le i\le N-3,\\
(3,2,1),&i\text{ is even and }2\le i\le N-2,\\
(4,3,2),&i=N-1.
\end{cases} \tag{4.7}
\]
For each $i$, the square between columns $i$ and $i+1$ has colors $x_i,y_i,d_i,d_{i+1}$, namely all four colors, and the three colors incident with each vertex are distinct. This includes the cyclic closure from column $N-1$ to column $0$, so (4.7) is a four-color B-coloring. The graph contains a $4$-cycle, so four colors are necessary.
\end{proof}

\begin{figure}[htbp]
\centering
\begin{tikzpicture}[x=0.9cm,y=1.25cm,scale=1.5,transform shape]
\foreach \i in {0,...,9}{\node[bvertex] (a\i) at (\i,1) {}; \node[bvertex] (b\i) at (\i,0) {};}
\foreach \i/\cx/\cy in {0/2/3,1/1/2,2/2/1,3/1/2,4/2/1,5/1/2,6/2/1,7/1/2,8/2/1}{
\draw[bedge\cx] (a\i)--node[edgecode,above]{\cx} (a\the\numexpr\i+1\relax);
\draw[bedge\cy] (b\i)--node[edgecode,below]{\cy} (b\the\numexpr\i+1\relax);
}
\draw[bedge3] (a9)
  .. controls (9,1.36) and (0,1.36)
  .. node[edgecode,above]{3} (a0);
\draw[bedge2] (b9)
  .. controls (9,-0.36) and (0,-0.36)
  .. node[edgecode,below]{2} (b0);
\foreach \i/\cd in {0/1,1/4,2/3,3/4,4/3,5/4,6/3,7/4,8/3,9/4}{\draw[bedge\cd] (a\i)--node[edgecode,right]{\cd} (b\i);}
\end{tikzpicture}
\caption{The four-color pattern (4.7) on $P_2\square C_{10}$. The same two-column pattern is repeated for every even cycle length, with the two prescribed seam columns.}
\label{fig:P2-even-cycle}
\end{figure}
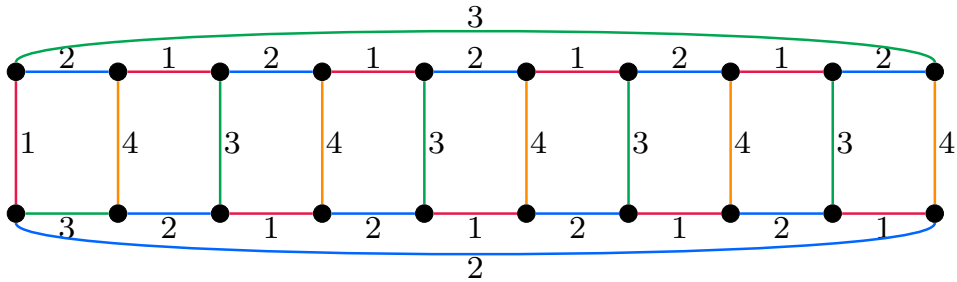

\begin{proof}[Proof of Theorem~\ref{thm:cylindrical-grids}]
The assertion for $C_{2n+1}\square P_m$ is Theorem~\ref{thm:odd-cycle-path}. Now let $n\ge2$. When $m=2$, Lemma~\ref{lem:P2-even-cycle} gives $q_B(C_{2n}\square P_m)=4$. Suppose that $m\ge3$. If $n=2$, then
\[
C_4\square P_m\cong P_2\square P_2\square P_m,
\]
and Theorem~\ref{thm:GS-path-products} gives $q_B(C_4\square P_m)=4$. If $n\ge3$, neither $C_{2n}$ nor $P_m$ contains a $4$-cycle, and each has a proper two-edge-coloring; hence each has a two-color B-coloring. Both factors have chromatic number two, and Theorem~\ref{thm:sabidussi} gives the same value for their product. Thus Lemma~\ref{lem:cross-inequality} applies and gives $q_B(C_{2n}\square P_m)\le4$. Since the product has maximum degree four, equality follows.
\end{proof}

\Needspace{10\baselineskip}
\section{Exceptional even discrete tori}
\label{sec:exceptional-even}
The key step is a six-color construction that is uniform in two arbitrary odd cycle lengths.

\begin{lemma}
\label{lem:exceptional-base}
For all integers $n_1,n_2\ge1$,
\[
q_B(P_2\square P_2\square C_{2n_1+1}\square C_{2n_2+1})=6.
\]
\end{lemma}

\begin{proof}
Set $r=2n_1+1$ and $q=2n_2+1$, and let
$G_{n_1,n_2}=P_2\square P_2\square C_{2n_1+1}\square C_{2n_2+1}$. We identify
\[
V(G_{n_1,n_2})=\{0,1\}^2\times\Z_r\times\Z_q.
\]
Represent each vertex as $(\varepsilon,\delta,x,y)$, where $\varepsilon,\delta\in\{0,1\}$, $x\in\Z_r$, and $y\in\Z_q$. Denote the four edge directions by
\[
\begin{aligned}
e_0(\delta,x,y)&=(0,\delta,x,y)(1,\delta,x,y),\\
e_1(\varepsilon,x,y)&=(\varepsilon,0,x,y)(\varepsilon,1,x,y),\\
f(\varepsilon,\delta,x,y)&=(\varepsilon,\delta,x,y)(\varepsilon,\delta,x+1,y),\\
g(\varepsilon,\delta,x,y)&=(\varepsilon,\delta,x,y)(\varepsilon,\delta,x,y+1).
\end{aligned}
\]
We use the color set $\Z_6$. Addition of a scalar to a four-tuple is componentwise in $\Z_6$. For fixed $(x,y)$, record the colors in the following order:
\[
\begin{aligned}
F_{x,y}&=(c(e_0(0,x,y)),c(e_1(1,x,y)),c(e_0(1,x,y)),c(e_1(0,x,y))),\\
X_{x,y}&=(c(f(0,0,x,y)),c(f(1,0,x,y)),c(f(0,1,x,y)),c(f(1,1,x,y))),\\
Y_{x,y}&=(c(g(0,0,x,y)),c(g(1,0,x,y)),c(g(0,1,x,y)),c(g(1,1,x,y))).
\end{aligned}
\]

The basic four-tuples are listed in Table~\ref{tab:exceptional-basic-blocks}. For $i\in\Z_3$, define $p_i=p+2i$, $\bar p_i=\bar p+2i$, $a_i=a+2i$, $b_i=b+2i$, $u_i=u+2i$, and $v_i=v+2i$.

\begin{table}[htbp]
\centering
\caption{Basic four-tuples for the six-color construction; all entries are in $\Z_6$.}
\label{tab:exceptional-basic-blocks}
\begin{tabular}{c@{\qquad}c|c@{\qquad}c}
\toprule
symbol & value & symbol & value\\
\midrule
$p$ & $(3,4,5,2)$ & $\bar p$ & $(4,5,2,3)$\\
$d_0$ & $(2,3,4,5)$ & $d_1$ & $(5,2,3,4)$\\
$d_2$ & $(3,0,1,2)$ & $d_3$ & $(1,4,5,0)$\\
$a$ & $(0,1,3,2)$ & $b$ & $(2,3,1,0)$\\
$u$ & $(1,2,0,3)$ & $v$ & $(3,0,2,1)$\\
\bottomrule
\end{tabular}
\end{table}

Let $s=(r-3)/2$, so $s\ge0$. For four-tuples $U,V$, let $(U,V)^s$ denote $s$ consecutive copies of the pair $U,V$; this part is omitted when $s=0$. We use four layer types, denoted by $S_0,S_1,S_2,S_3$. If the $y$th $C_q$-layer has type $S_i$, its $F$- and $X$-profiles are the two words in row $S_i$ of Table~\ref{tab:exceptional-states}, read in the order $x=0,1,\ldots,r-1$.

\begin{table}[htbp]
\centering
\caption{The four layer types. Each profile has length $r=2n_1+1=2s+3$.}
\label{tab:exceptional-states}
\begin{tabular}{c|c|c}
\toprule
type & $(F_{x,y})_{x\in\Z_r}$ & $(X_{x,y})_{x\in\Z_r}$\\
\midrule
$S_0$ & $(p_0,p_1,p_2,(p_0,d_0)^s)$ & $(a_0,a_1,a_2,(a_0,a_2)^s)$\\
$S_1$ & $(\bar p_0,\bar p_1,\bar p_2,(\bar p_0,d_1)^s)$ & $(b_0,b_1,b_2,(b_0,b_2)^s)$\\
$S_2$ & $(\bar p_2,\bar p_0,\bar p_1,(\bar p_2,d_2)^s)$ & $(b_2,b_0,b_1,(b_2,b_1)^s)$\\
$S_3$ & $(\bar p_1,\bar p_2,\bar p_0,(\bar p_1,d_3)^s)$ & $(b_1,b_2,b_0,(b_1,b_0)^s)$\\
\bottomrule
\end{tabular}
\end{table}

Table~\ref{tab:exceptional-transitions} gives the colors on the edges joining two consecutive layers. If a layer of type $S_i$ is followed by one of type $S_j$, the corresponding row prescribes the $Y$-profile in the order $x=0,1,\ldots,r-1$.

\begin{table}[htbp]
\centering
\caption{The $Y$-profiles on edges joining consecutive layers.}
\label{tab:exceptional-transitions}
\begin{tabular}{c|c}
\toprule
consecutive layer types & $(Y_{x,y})_{x\in\Z_r}$\\
\midrule
$S_0\to S_1$, $S_3\to S_1$ & $(u_0,u_1,u_2,(u_0,v_0)^s)$\\
$S_1\to S_0$, $S_1\to S_2$ & $(u_2,u_0,u_1,(u_2,v_2)^s)$\\
$S_2\to S_3$ & $(u_1,u_2,u_0,(u_1,v_1)^s)$\\
\bottomrule
\end{tabular}
\end{table}

Let $t=(q-3)/2$, so $t\ge0$. Assign the following types to the $C_q$-layers in cyclic order:
\[
S_1,\quad
\underbrace{S_0,S_1,\ldots,S_0,S_1}_{t\text{ copies of the pair }S_0,S_1},
\quad S_2,\quad S_3.
\]
The repeated part is omitted when $t=0$. Every consecutive pair appears in Table~\ref{tab:exceptional-transitions}, including $S_3,S_1$ at the cyclic closure, so the table completes the definition of the coloring. Figure~\ref{fig:exceptional-state-digraph} illustrates this layer order.

In each profile, the first three entries are prescribed separately and the rest alternate between two fixed four-tuples. It therefore suffices to check consecutive $x$-coordinates among the first three positions, at the junction with the alternating part, within its two possible adjacencies, and at the cyclic closure; these checks are independent of $s$.

For a vertex $(\varepsilon,\delta,x,y)$, substitution in Table~\ref{tab:exceptional-states} shows that the colors on
\[
e_0(\delta,x,y),\quad e_1(\varepsilon,x,y),\quad
f(\varepsilon,\delta,x-1,y),\quad f(\varepsilon,\delta,x,y)
\]
are distinct. The same substitutions show that the three types of coordinate $4$-cycle contained in one layer are rainbow. Next, for each row of Table~\ref{tab:exceptional-transitions}, the color of a joining $g$-edge differs from the four non-$g$ colors at both endpoints, and the three types of coordinate $4$-cycle containing that edge are rainbow.

It remains to compare the two $Y$-profiles entering and leaving a layer. The cyclic order has at most seven local three-layer patterns: two in the repeated part and the others at its ends or around $S_1,S_2,S_3$. Entrywise comparison in Table~\ref{tab:exceptional-transitions} shows that the incoming and outgoing $g$-edges have different colors in every case.

Thus the six edges incident with every vertex have pairwise distinct colors. Since each factor is $4$-cycle-free, every $4$-cycle of $G_{n_1,n_2}$ is a coordinate square and hence is rainbow by the preceding checks. This gives a six-color B-coloring. Since $\Delta(G_{n_1,n_2})=6$, properness gives the matching lower bound.
\end{proof}

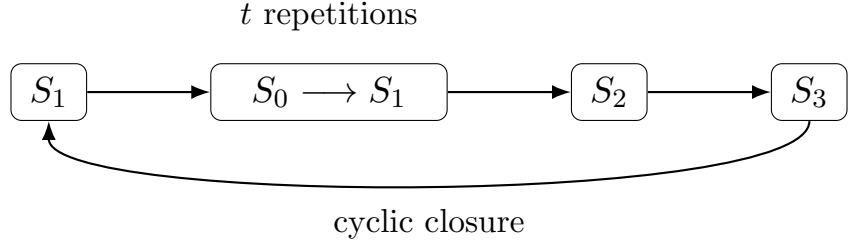
\begin{figure}[htbp]
\centering
\begin{tikzpicture}[node distance=13mm,scale=1.25,transform shape]
\node[statebox] (S1) {$S_1$};
\node[statebox,right=of S1,minimum width=25mm] (repeat) {$S_0\longrightarrow S_1$};
\node[above=2mm of repeat,font=\footnotesize] {$t$ repetitions};
\node[statebox,right=of repeat] (S2) {$S_2$};
\node[statebox,right=of S2] (S3) {$S_3$};
\draw[-{Latex[length=2.5mm]},thick] (S1) -- (repeat);
\draw[-{Latex[length=2.5mm]},thick] (repeat) -- (S2);
\draw[-{Latex[length=2.5mm]},thick] (S2) -- (S3);
\draw[-{Latex[length=2.5mm]},thick] (S3.south)
  .. controls +(0,-9mm) and +(0,-9mm)
  .. node[below=1mm,font=\footnotesize]{cyclic closure} (S1.south);
\end{tikzpicture}
\caption{The cyclic order of the $C_q$-layers. Starting from $S_1$, the pair $S_0,S_1$ is repeated $t=(q-3)/2$ times, followed by $S_2,S_3$ and the return to $S_1$. The repeated part is absent when $t=0$.}
\label{fig:exceptional-state-digraph}
\end{figure}

\begin{corollary}
\label{cor:exceptional-triple}
For all integers $n_1,n_2\ge1$,
\[
q_B(C_4\square C_{4n_1+2}\square C_{4n_2+2})=6.
\]
\end{corollary}

\begin{proof}
By Lemma~\ref{lem:cover-pullback}, repeating the coloring in Lemma~\ref{lem:exceptional-base} twice in each cyclic coordinate gives a six-color B-coloring of $P_2\square P_2\square C_{4n_1+2}\square C_{4n_2+2}$. Since $P_2\square P_2\cong C_4$, this gives the required upper bound. The lower bound is the maximum degree.
\end{proof}

\begin{proof}[Proof of Theorem~\ref{thm:exceptional-even-classification}]
The cases $d=1$ and $d=2$ are Theorem~\ref{thm:GS-even-products}. Suppose $d\ge3$. Since the discrete torus is exceptional, $C_{2n_1}=C_4$ and $n_2,\ldots,n_d$ are odd. By Corollary~\ref{cor:exceptional-triple},
\[
q_B(C_4\square C_{2n_2}\square C_{2n_3})=6.
\]
Every remaining even cycle has chromatic number and B-coloring number equal to two. By Theorem~\ref{thm:sabidussi}, the three-factor product above and every partial product obtained by adjoining further even cycles also have chromatic number two. Hence Lemma~\ref{lem:cross-inequality} applies repeatedly to the six-color base block and the two-color cycle factors, giving
\[
q_B(C_{2n_1}\square\cdots\square C_{2n_d})\le6+2(d-3)=2d.
\]
The product is $2d$-regular, so equality follows.
\end{proof}

\section{Even \texorpdfstring{$\ell$}{ell}-cylindrical grids}
\label{sec:path-products}

Together, Theorems~\ref{thm:GS-even-products} and~\ref{thm:exceptional-even-classification} yield the following classification.

\begin{corollary}
\label{cor:all-even-products}
Let $E$ be a Cartesian product of even cycles. Then $q_B(E)=\Delta(E)$ unless $E\cong C_4$ or $E\cong C_4\square C_{4n+2}$ for some integer $n\ge1$. In the two exceptional cases, $q_B(C_4)=4$ and $q_B(C_4\square C_{4n+2})=5$.
\end{corollary}

\begin{proof}[Proof of Theorem~\ref{thm:path-even-cycle-products}]
At a vertex of maximum degree, properness requires $\Delta(G)$ distinct colors; hence $q_B(G)\ge\Delta(G)$.

Item (1) is Lemma~\ref{lem:P2-even-cycle}. Item (2) follows from Theorem~\ref{thm:GS-path-products}, since $P_2\square C_4\square C_4\cong Q_5$.

Every path and even-cycle factor has chromatic number two, and Theorem~\ref{thm:sabidussi} gives the same value for every nontrivial product formed from them. Since all B-colorings used below have at least two colors, Lemma~\ref{lem:cross-inequality} applies whenever invoked.

We next prove the upper bounds in item (3). In (3a), select one cycle factor of $E$. Lemma~\ref{lem:P2-even-cycle} colors its product with $P_2$ using four colors, while the remaining even-cycle product has B-coloring number equal to its maximum degree because it contains no $C_4$ factor. Lemma~\ref{lem:cross-inequality} now gives $q_B(G)\le\Delta(G)+1$.

In (3b), apply the same lemma to $P_2\square C_4\square C_4\cong Q_5$ and $C_{4n+2}$ to obtain the upper bound $6+2=\Delta(G)+1$.

In (3c), Theorem~\ref{thm:GS-path-products} gives $q_B(P_2\square P_m)=4=\Delta(P_2\square P_m)+1$, while $q_B(E)=\Delta(E)$ because $E$ has no $C_4$ factor. Another application of Lemma~\ref{lem:cross-inequality} gives $\Delta(G)+1$.

It remains to prove item (4). Let $k=c_4(E)$.

First suppose that $P$ is neither $P_2$ nor $P_2\square P_m$. If $E$ is not one of the two exceptions in Corollary~\ref{cor:all-even-products}, apply Lemma~\ref{lem:cross-inequality} directly to $P$ and $E$. If $E\cong C_4$, regard $P\square C_4$ as a product of paths by replacing $C_4$ with $P_2\square P_2$; the resulting path product is not one of the exceptions in Theorem~\ref{thm:GS-path-products}. If $E\cong C_4\square C_{4n+2}$, first color $P\square C_4$ in this way and then add the remaining cycle using Lemma~\ref{lem:cross-inequality}. In each case the number of colors is $\Delta(G)$.

Suppose next that $P\cong P_2\square P_m$ and $k\ge1$. Absorb all $k$ copies of $C_4\cong P_2\square P_2$ into the path part. The resulting product $P_2^{2k+1}\square P_m$ is not a path-product exception, while the remaining even-cycle product contains no $C_4$ factor. Theorem~\ref{thm:GS-path-products} gives the required coloring if no even-cycle factor remains. Otherwise, Theorems~\ref{thm:GS-path-products} and~\ref{thm:GS-even-products}, followed by Lemma~\ref{lem:cross-inequality}, give $q_B(G)=\Delta(G)$.

It remains to consider $P\cong P_2$. The cases $t=1$, $E\cong C_4^2$, $k=0$ with $t\ge2$, and $E\cong C_4^2\square C_{4n+2}$ have already been listed. In the next two cases, $H$ is connected, has at least two edges, and Theorem~\ref{thm:sabidussi} gives $\chi(H)=2$; hence it is bipartite, as required in Lemma~\ref{lem:Q3-product}.

If $k=1$, let $H$ be the nonempty product obtained by deleting the unique $C_4$ factor, so that $G=Q_3\square H$. Since $H$ contains no $C_4$ factor, Corollary~\ref{cor:all-even-products} and Lemma~\ref{lem:Q3-product} give $q_B(G)=\Delta(G)$.

If $k=2$ and the graph is not one of the listed cases, retain one $C_4$ factor in $H$ and again express $G$ as $Q_3\square H$. The product $H$ is neither $C_4$ nor $C_4\square C_{4n+2}$, so Corollary~\ref{cor:all-even-products} and Lemma~\ref{lem:Q3-product} apply.

Finally, if $k\ge3$, absorb all $C_4$ factors into the path part, obtaining the hypercube $Q_{2k+1}$. Since $2k+1\ge7$, Theorem~\ref{thm:GS-path-products} gives $q_B(Q_{2k+1})=2k+1$. This is the desired coloring when no even-cycle factor remains. Otherwise, the remaining even-cycle product contains no $C_4$ factor, and Lemma~\ref{lem:cross-inequality} completes the proof.
\end{proof}

\section{Open problems}
Theorems~\ref{thm:torus-grids} and~\ref{thm:exceptional-even-classification} exhibit two different parity phenomena. A single odd factor forces five colors in a two-dimensional cycle product, whereas every exceptional even discrete torus attains the degree lower bound once its dimension is at least three. Corollary~\ref{cor:many-odd-upper} also gives, for every positive integer $d$ and every product $G$ of $d$ odd cycles,
\[
2d+1\le q_B(G)\le\left\lceil\frac{5d}{2}\right\rceil.
\]
The lower bound is attained under the divisibility hypothesis in Theorem~\ref{thm:GS-odd-products}. This suggests the following conjecture.

\begin{conjecture}
For all integers $d\ge1$ and $n_1,\ldots,n_d\ge1$,
\[
q_B(C_{2n_1+1}\square\cdots\square C_{2n_d+1})=2d+1.
\]
\end{conjecture}

The three interval classes in Theorem~\ref{thm:path-even-cycle-products}(3) are the remaining unresolved even $\ell$-cylindrical grids. Determining which of these grids have B-coloring number $\Delta$ and which require $\Delta+1$ is the second open problem arising from our classification.

\section*{Declaration of AI usage}
During the development and preparation of this manuscript, the authors used ChatGPT to explore possible approaches to selected parts of the mathematical arguments, and to improve the language and presentation of the paper. All AI-assisted arguments were carefully checked, revised, and independently verified by the authors. The authors take full responsibility for the correctness, originality, and final content of the manuscript.

\appendix

\section{Seam data for the odd--odd construction}
\label{app:odd-seams}
This appendix gives the finite lists used in Proposition~\ref{prop:odd-odd-upper}. For the integers $n_1,n_2\ge1$ in that proposition, set $p=2n_1+1$ and $q=2n_2+1$. Let $h^{\circ}(i,j)$ and $v^{\circ}(i,j)$ denote the replacement values assigned to $h_0(i,j)$ and $v_0(i,j)$ in the terminal rows. Each ordered pair in Table~\ref{tab:terminal-row-corrections} is $(h^{\circ}(i,j),v^{\circ}(i,j))$, and the row is chosen according to the parity of $i$.

\begingroup
\captionsetup{font=small,skip=3pt,hypcap=false}
\setlength{\tabcolsep}{5pt}
\renewcommand{\arraystretch}{0.94}

\Needspace{15\baselineskip}
\begin{center}
\captionof{table}{Corrections in the terminal rows of the periodic coloring.}
\label{tab:terminal-row-corrections}
\small
\begin{tabular}{c|c|ccc}
\toprule
$q\pmod6$ & $i\pmod2$ & \multicolumn{3}{c}{terminal rows, from left to right}\\
\midrule
1 & 0 & $(2,3)$ at $q-2$ & $(4,1)$ at $q-1$ & \\
1 & 1 & $(5,1)$ at $q-2$ & $(2,3)$ at $q-1$ & \\
\midrule
3 & 0 & $(3,1)$ at $q-3$ & $(2,5)$ at $q-2$ & $(4,1)$ at $q-1$\\
3 & 1 & $(4,5)$ at $q-3$ & $(3,1)$ at $q-2$ & $(2,3)$ at $q-1$\\
\midrule
5 & 0 & $(3,1)$ at $q-2$ & $(4,2)$ at $q-1$ & \\
5 & 1 & $(5,2)$ at $q-2$ & $(3,1)$ at $q-1$ & \\
\bottomrule
\end{tabular}
\end{center}

For the three terminal columns, recall that
\[
T_j=(H_j^0,H_j^{p-2},H_j^{p-1};V_j^0,V_j^{p-2},V_j^{p-1}).
\]
The two short cases are listed in Table~\ref{tab:short-T-lists}.

\Needspace{8\baselineskip}
\begin{center}
\captionof{table}{The terminal-column lists for $q=7$ and $q=9$.}
\label{tab:short-T-lists}
\footnotesize
\renewcommand{\arraystretch}{0.90}
\begin{tabular}{c|c|c}
\toprule
$j$ & $T_j$ for $q=7$ & $T_j$ for $q=9$\\
\midrule
0 & $(1,1,2;3,2,4)$ & $(1,1,2;3,2,4)$\\
1 & $(4,5,1;5,3,2)$ & $(4,5,1;5,3,2)$\\
2 & $(2,4,3;4,1,5)$ & $(2,4,3;4,1,5)$\\
3 & $(3,3,1;5,2,4)$ & $(3,3,1;5,2,4)$\\
4 & $(1,1,3;4,4,2)$ & $(1,1,2;4,3,5)$\\
5 & $(2,3,5;3,1,4)$ & $(2,2,1;5,1,3)$\\
6 & $(5,2,1;4,3,5)$ & $(3,4,2;4,5,1)$\\
7 & --- & $(2,3,5;3,1,4)$\\
8 & --- & $(5,2,1;4,3,5)$\\
\bottomrule
\end{tabular}
\end{center}

For $q\ge11$, let
\[
E_j=\begin{cases}5,&j\text{ even},\\4,&j\text{ odd},\end{cases}
\qquad
A=(3,1,2),\quad B=(2,3,1),\quad C=(1,2,3),
\]
where the entries of $A$, $B$, and $C$ are indexed by $\Z_3$. In every indicated middle range, set
\[
T_j=(E_j,A_{j\bmod3},B_{j\bmod3};C_{j\bmod3},B_{j\bmod3},E_j). \tag{A.1}
\]
The initial lists are given in Table~\ref{tab:long-T-initial}; the terminal lists are given in Table~\ref{tab:long-T-terminal}.

\begin{center}
\begin{minipage}{0.82\textwidth}
\centering
\captionof{table}{Initial terminal-column lists for $q\ge11$.}
\label{tab:long-T-initial}
\small
\begin{tabular}{c|c|c}
\toprule
$j$ & $q\equiv1,3\pmod6$ & $q\equiv5\pmod6$\\
\midrule
0 & $(1,1,2;3,2,4)$ & $(3,1,2;1,2,4)$\\
1 & $(4,5,1;5,3,2)$ & $(4,5,3;2,3,1)$\\
2 & $(2,4,3;4,1,5)$ & $(5,2,4;3,1,5)$\\
3 & $(3,3,1;5,2,4)$ & $(4,3,1;5,2,4)$\\
4 & $(1,1,2;4,3,5)$ & $(1,1,2;4,3,5)$\\
5 & $(2,2,1;3,1,4)$ & $(2,2,1;3,1,4)$\\
\bottomrule
\end{tabular}
\end{minipage}
\end{center}

\Needspace{14\baselineskip}
\begin{center}
\captionof{table}{Middle ranges and terminal-column lists for $q\ge11$.}
\label{tab:long-T-terminal}
\small
\begin{tabular}{c|c|c|l}
\toprule
$q\pmod6$ & middle range for (A.1) & index & $T_j$\\
\midrule
1 & $6\le j\le q-4$ & $q-3$ & $(5,1,3;4,4,2)$\\
  &                    & $q-2$ & $(2,3,5;3,1,4)$\\
  &                    & $q-1$ & $(5,2,1;4,3,5)$\\
\midrule
3 & $6\le j\le q-5$ & $q-4$ & $(4,2,1;3,5,4)$\\
  &                    & $q-3$ & $(2,1,5;1,4,2)$\\
  &                    & $q-2$ & $(4,5,3;2,1,4)$\\
  &                    & $q-1$ & $(5,2,1;4,3,5)$\\
\midrule
5 & $6\le j\le q-4$ & $q-3$ & $(5,2,1;3,4,5)$\\
  &                    & $q-2$ & $(4,1,2;1,5,4)$\\
  &                    & $q-1$ & $(5,2,3;4,3,5)$\\
\bottomrule
\end{tabular}
\end{center}

\endgroup
\FloatBarrier
\printbibliography

\end{document}